\documentclass[10pt,reqno]{amsart}
\usepackage{palatino,lipsum}
\usepackage[mathlines,pagewise]{lineno}
\usepackage{amsfonts,amsmath} 
\usepackage{amsthm}
\usepackage{graphicx,enumerate,soul}

\usepackage[margin=1in,hmarginratio=1:1]{geometry}
\usepackage[latin1]{inputenc}
\usepackage[all]{xy}

\usepackage{amsmath}
\usepackage{amssymb,mathtools}
\usepackage{stmaryrd}
\usepackage{hyperref}
\usepackage{pdfsync}
\usepackage{color}
\usepackage[T1]{fontenc} 

\makeatletter
\let\start@align@nopar\start@align
\let\start@gather@nopar\start@gather
\let\start@multline@nopar\start@multline
\long\def\start@align{\par\start@align@nopar}
\long\def\start@gather{\par\start@gather@nopar}
\long\def\start@multline{\par\start@multline@nopar}
\makeatother

\newtheorem{counter}{Counter}
\newtheorem{lem}[counter]{Lemma}
\newtheorem{defn}[counter]{Definition}
\newtheorem{thm}[counter]{Theorem}
\newtheorem{remark}[counter]{Remark}

\newcommand{\R}{\mathbb{R}}
 
\renewcommand{\H}{\mathcal{H}} %Bbb
\renewcommand{\L}{\mathcal{L}}

\newcommand{\U}{\mathcal{U}} 

\newcommand{\F}{\mathcal{F}}
\newcommand{\V}{\mathcal{V}}

\renewcommand{\lg}{\langle} 
\newcommand{\rg}{\rangle} 
\newcommand{\lra}{\longrightarrow} 

\newcommand{\ra}{\rightarrow} 

\newcommand{\red}{\color{black}}

 \newcommand{\sse}{\subseteq}

\renewcommand{\~}{\tilde}

\newcommand{\rhu}{\rightharpoonup}

\newcommand{\fal}{\forall}
\newcommand{\8}{\infty} 

\newcommand{\vph}{\varphi}
\newcommand{\vep}{\varepsilon} %lettere greche
 
\newcommand{\dt}{\delta} 
\newcommand{\al}{\alpha}

 \newcommand{\ggm}{\Gamma}

 \newcommand{\vth}{\vartheta}

\newcommand{\disp}{\displaystyle}

\DeclareMathOperator{\dd}{\operatorname{d}\!}

\DeclareMathOperator{\loc}{{\operatorname{loc}}}
\DeclareMathOperator{\per}{{\operatorname{per}}}

{   \end{list} }
\newcommand{\magenta}{\color{black}}
\newcommand{\blue}{\color{black}}
\definecolor{mygreen}{rgb}{0.1,0.75,0.2}

\title[Epitaxial Growth]{Regularity in time for weak solutions of a continuum model
for epitaxial growth with elasticity on vicinal surfaces}

\author{Irene Fonseca, Giovanni Leoni and Xin Yang Lu}

\date{}

\begin{document}

\begin{abstract}
{
The evolution equation derived by Xiang (SIAM J. Appl. Math. 63:241--258, 2002) 
to describe vicinal surfaces in heteroepitaxial growth}
is
\begin{equation}\label{abs eq}
h_t=-\left[ H(h_x)+\left(h_x^{-1}+h_x  \right) h_{xx}\right]_{xx},
\end{equation}
where $h$ denotes the surface height of the film, and $H$ { is} the Hilbert transform.
{ Existence of solutions was} obtained by Dal Maso,
Fonseca and Leoni (Arch. Rational Mech. Anal. 212: 1037--1064, 2014).
% where the authors transformed
%\eqref{abs eq} into the parabolic equation
%\begin{equation}\label{abs eq2}
%u_t=-\left[H(u_{x})+\log(u_{xx}+a)+(u_{xx}+a)^2/2 \right]_{xx}
%%\quad \Phi(\xi):=
%%\left\{
%%\begin{array}{cl}
%%+\8 & \text{if } \xi<0,\\
%%0 & \text{if } \xi=0,\\
%%\xi\log\xi+\xi^3/6 & \text{if } \xi>0.
%%\end{array}\right.
%\end{equation}
%for some constant $a>0$, and proved the existence of weak solutions, to be intended as solutions
%of some particular variational inequalities. 
{ T}he regularity in time { was left unresolved}. 
The aim of this paper is to prove existence, uniqueness, and Lipschitz regularity in time for
weak solutions, under suitable assumptions on the initial datum.
%
%In this paper we consider the model derived by Xiang, and prove the existence of Lipschitz regular
%weak solutions.
%
%
%
%Discrete models for heteroepitaxial growth organizing according to misfit elasticity forces
%have been proposed by Duport, Politi and Villain (J. Phys. I 5:1317--1350, 1995), and
%Tersoff, Phang, Zhang and Lagally (Phys. Rev. Lett. 75:2730--2733, 1995). A continuum model
%has been derived by Xiang (SIAM J. Appl. Math. 63:241--258, 2002) (see also
%Xiang and E (Phys. Rev. B 69:035409-1--035409-16, 2004) and Xu and Xiang (SIAM J. App. Math.
%69:1393--1414, 2009)). 
%Analytical validation of such model has been provided by Dal Maso,
%Fonseca and Leoni (Arch. Rational Mech. Anal. 212: 1037--1064, 2014), 
%proving the existence of weak solutions.
%In this paper we consider the model derived by Xiang, and prove the existence of Lipschitz regular
%weak solutions.
\end{abstract}

\maketitle

{\bf Keywords}: epitaxial growth, vicinal surfaces, evolution equations, Hilbert
transform, monotone operators

\medskip

{\bf AMS Mathematics Subject Classification}: 35K55, 35K67, 44A15, 74K35  

\bigskip
%\newpage
%\setpagewiselinenumbers
%\linenumbers

\section{Introduction}

Within the context of heteroepitaxial growth of a film onto a substrate,
terraces and steps self-organize {to accommodate} misfit elasticity forces. Discrete models
have been proposed by Duport, Politi and Villain \cite{DPV}, and Tersoff, Phang, Zhang and
Lagally \cite{TPZL}. A continuum variant of these models has been derived
by Xiang \cite{X}. Also related are the works by Xiang and E \cite{XE},
 and Xu and Xiang \cite{XX}. 
 The evolution equation derived by Xiang \cite[Formula~(3.62)]{X} is
 (upon space inversion)
\begin{equation}\label{eqX}
 h_t=-\left[ H(h_x)+\left(\frac1{h_x}+h_x  \right) h_{xx}\right]_{xx},
 \end{equation}
 where $h$ describes the height of the surface of the film, 
 { and} is assumed { to be} monotone. The time
 domain is $[0,T]$ with $T>0$ a given datum, the space domain is $I:=(-\pi,\pi)$, 
$H$ {denotes} the Hilbert transform, i.e., 
$$H(f)(x):=\frac1{2\pi} PV \int_I \frac{f(x-y)}{\tan (y/2)}\dd y, $$
with $PV$ denoting the Cauchy principal value.
 Analytical validation for the continuum model
from \cite{X} has been {obtained} by Dal Maso, Fonseca and Leoni {in} \cite{DFL}, {where} the authors
transform \eqref{eqX} into a parabolic evolution equation
\begin{equation}\label{origin}
u_t=-\left[H(u_{x})+\Phi_a'(u_{xx}) \right]_{xx},
\end{equation}
$$\Phi_a(\xi):=\Phi(\xi+a),\qquad
\Phi:\R\lra \R\cup \{+\8\},
\quad \Phi(\xi):=
\left\{
\begin{array}{cl}
+\8 & \text{if } \xi<0,\\
0 & \text{if } \xi=0,\\
\xi\log\xi+\xi^3/6 & \text{if } \xi>0.
\end{array}\right.$$
Here $a>0$ is a constant, and $u$ is a suitable antiderivative of $h$.
{ T}he main results {in \cite{DFL} is} the {proof of the} existence of weak solutions for \eqref{origin}
in the sense that:
\begin{enumerate}
\item (\cite[Theorem~1]{DFL}) 
for any $T,a>0$, $u^0\in L_{\per_0}^2(I)$, there exists $u\in L^3(0,T;W_{\per_0}^{2,3}(I))$
such that
\begin{align*}
\int_0^T \int_I \big[w_t(t)(w(t)-u(t)) &- H(u_{xx}(t))(w_x(t)-u_x(t)) \\
&+
\Phi_a(w_{xx}(t)) -\Phi_a(u_{xx}(t))\big] \dd x \dd t \geq 0
\end{align*}
for any test function $w\in L^3(0,T;W_{\per_0}^{2,3}(I))$
such that $w_t\in L^{3/2}(0,T;(W_{\per_0}^{2,3}(I))')$ and $w(0)=u^0$. 
Moreover{,} $\log(u_{xx}+a)\in L^1(0,T;L^1(I))$;

\item (\cite[Theorem~2]{DFL}) assuming, in addition, that test functions $w$ satisfy
$\log(w_{xx}+a)\in L^{3/2}(0,T;L^{3/2}(I))$, it holds
\begin{align*}
\int_0^T \int_I \big[w_t(t)(w(t)-u(t)) &- H(u_{xx}(t))(w_x(t)-u_x(t)) \\
&+
\Phi_a'(w_{xx}(t))(w_{xx}(t)-u_{xx}(t))\big] \dd x\dd t\leq 0 .
\end{align*}
Here
$$W_{\per_0}^{2,3}(I):=\left\{f\in W_{\loc}^{2,3}(\R):f \text{ is } 2\pi\text{-periodic and }
 \int_I f \dd x=0 \right\},$$
$$L_{\per_0}^{2}(I):=\left\{f\in L_{\loc}^{2}(\R):
f \text{ is } 2\pi\text{-periodic and } \int_I f \dd x=0 \right\}.$$
\end{enumerate}
Note in that both results, the regularity in time
  is assumed on the test function $w$. Concerning the regularity in time
 of $u$, it {was} only { proved} (\cite[Remark~3]{DFL}) that $u$
 has finite essential pointwise variation when considered as function
 $u:[0,T]\lra (W_{\per_0}^{2,\8}(I))'$, where
 $$W_{\per_0}^{2,\8}(I):=\left\{f\in W_{\loc}^{2,\8}(\R):f \text{ is } 2\pi\text{-periodic and }
 \int_I f \dd x=0 \right\}.$$
 The main result of this paper is: 
\begin{thm}\label{main}
Given $T,a>0$, {and} $u^0\in W^{2,2}_{\per_0}(I)$ such that 
%{ there exists $z^0\in U$ such that}
\begin{equation}\label{id}
{\int_I z^0v \dd x} - \int_I H(u^0_{xx}) v_x \dd x + \int_I [\Phi_a(v_{xx})-\Phi_a(u_{xx}^0)] \dd x\geq 0
\end{equation}
{for some $z^0\in { L_{\per_0}^2(I)}$ and any $v\in { W_{\per_0}^{2,3}(I)}$, then} there exists a solution
$u:[0,T]\lra W_{\per_0}^{2,3}(I)$ of \eqref{origin} in the sense that
\begin{equation}\label{3*}
\int_0^T \int_I u_t(t) \vph(t,x) \dd x \dd t =
\int_0^T \int_I \big[H(u_{xx}(t)) \vph_x(t,x)- \Phi_a'(u_{xx}(t))\vph_{xx}(t,x)\big] \dd x \dd t 
\end{equation}
for any $\vph\in C_c^\8((0,T)\times I;\R)$. Moreover{,} 
$$u\in L^\8(0,T;{\magenta W_{\per_0}^{2,3}(I)})\cap 
C^0([0,T];{\magenta L_{\per_0}^2(I)}), \qquad u_t\in L^\8(0,T;{\magenta L_{\per_0}^2(I)}),\qquad
u(0)=u^0.$$
\end{thm}
The { main} argument is to first prove that 
{the }variational inequality \eqref{vi} below admits a solution $u$, {and}
then show that such $u$ is also solution of \eqref{origin} in the sense of Theorem \ref{main}.
We remark that 
%condition \eqref{id} is non empty, that is 
there { is a large class of initial data} $u^0$ satisfying \eqref{id}.
Assume { that} $u_{xx}^0+a>0$ a.e., and { that} $\Phi_a'(u_{xx}^{\red0})$, $\Phi_a(u_{xx}^{\red0})\in L^1(I)$.
Then the convexity of $\Phi$ gives
$$\int_I [\Phi_a(v_{xx})-\Phi_a(u_{xx}^0)]\dd x\geq \int_I (v_{xx}-u_{xx}^0)\Phi_a'(u_{xx}^0)\dd x,$$
thus a sufficient condition for \eqref{id} is { that, for some
$z^0\in { L_{\per_0}^2(I)}$ and any $v\in { W_{\per_0}^{2,3}(I)}$,}
$$\int_I z^0 v\dd x - \int_I H(u^0_{xx}) v_x \dd x +\int_I v_{xx}\Phi_a'(u_{xx}^0)\dd x-
\int_I u_{xx}^0 \Phi_a'(u_{xx}^0)  \dd x\geq 0.$$
{ In particular, the previous inequality holds if}
$$\int_I u_{xx}^0 \Phi_a'(u_{xx}^0)\dd x\leq 0,\qquad z^0:=[-H(u_x^{0})-\Phi_a'(u_{xx}^{0})]_{xx}{.}$$
%with the last inequality to be intended in the weak sense (i.e., tested against functions of $V$).
%Note that we do {\em not} require $u_x\in W_{\loc}^{1,3}(\R)$.
%if $u$ is sufficiently regular such that $[-H(u_x)-\Phi_a'(u_{xx})]_{xx}$ is 
%well-defined a.e. as {\em function} (e.g., 
{ Observe that} if ${\magenta u}\in C^4(I)$, with derivatives bounded away from $0$,
{ and} extended by periodicity, 
then such a $z^0$ { is well defined}. 

{ To ensure that} $\int_I u_{xx}^0 \Phi_a'(u_{xx}^0)\dd x\leq 0$,
the following { are sufficient conditions:}
\begin{enumerate}
\item if $\Phi_a'(0)\geq 0$, then due to the monotonicity of $\Phi_a'$, there exists a unique 
$b_0\leq 0$ such that $\Phi_a'(b_0)=0$. Thus any $u^0$ with $b_0\leq u^0_{xx}\leq 0 $ is acceptable;

\item similarly, if $\Phi_a'(0)\leq0$, then there exists a unique 
$b_1\geq 0$ such that $\Phi_a'(b_1)=0$. Thus any $u^0$ with $0\leq u^0_{xx}\leq b_1 $ is acceptable.
\end{enumerate}

\section{Proof of { Theorem \ref{main}}}
Let $T>0$ be given, and { let} $I:=(-\pi,\pi)$ be the space domain.
Let 
\begin{equation}\label{space}
V:=W_{\per_0}^{2,3}(I), \quad U:=L^2_{\per_0}(I), \quad \V:=L^2(0,T;V),
\quad \U:=L^2(0,T;U).
\end{equation}
Note that $U$ is an Hilbert space, $V$ is a reflexive Banach space, and the embedding
$V\hookrightarrow U$ is {  compact}. Duality yields the pivot space structures
\begin{equation}\label{7*}
V\hookrightarrow U \hookrightarrow V',\qquad \V\hookrightarrow \U
\hookrightarrow \V'.
\end{equation}
For future reference, {$\lg, \rg$} (resp. $\lg, \rg_{V',V}$) will denote the
duality pairing between {$L^2(I)$ and $L^2(I)$} (resp. $V'$ and $V$).

{\blue
\begin{defn}
An operator $A:V\lra V'$ is:
\begin{enumerate}
\item {\bf monotone} if for any $u,v\in V$, it holds
$$\lg Au-Av,u-v\rg_{V',V} \geq 0. $$
Similarly, a set $G\sse V\times V'$ is {\em ``monotone''} if for any pair $(u,u')$, $(v,v')\in G$,
it holds 
$$\lg u'-v',u-v\rg_{V',V} \geq 0.$$

\item {\bf maximal monotone} if the graph 
$$\ggm_A:=\{(u,Au):u\in V\}\sse V\times V'$$
is not a proper subset of any monotone set.

\item {\bf pseudo-monotone} if it is bounded, and
 $$\lg Au, u-v \rg_{V',V} \leq \liminf_n \lg Au^n,u^n-v\rg_{V',V}$$  
 for every $v\in V$,
 $u^n,u\in V$, satisfying $u^n\rhu u$ and $\limsup_n \lg Au^n,u^n-u\rg_{V',V} \leq 0$.

\item {\bf hemi-continuous} if for any $u,v\in V$ the mapping
$ t\longmapsto \lg A{\red(}u+tv{\red)},v\rg_{V',V} $ is continuous.
\end{enumerate}
\end{defn}

}

{\magenta
\begin{remark}\label{2*}
If an operator $A:V\to V'$ is monotone and hemi-continuous, then it is maximal monotone
(see \cite[Theorem~1.2]{B}).
\end{remark}
}
%
%\begin{defn}
%An operator $A:V\lra V'$ is ``pseudo-monotone'' if $A$ is bounded{ ,} and
%{ $\lg Au, u-v \rg_{V',V} \leq \liminf_n \lg Au^n,u^n-v\rg_{V',V}$  for every $v\in V$
%whenever $u^n,u\in V$ satisfy $u^n\rhu u$ and $\limsup_n \lg Au^n,u^n-u\rg_{V',V} \leq 0$}.
%\end{defn}
%
{ We will use the following result (see Ka\v{c}ur \cite{Kacur}).}
%
%%We use a very powerful result (a special case of \cite[Theorem~6.1]{Rudd}): 
%%\footnote{
%%As remarked in \cite{Rudd}, a more general result has been proven in \cite{Kacur},
%%where $A$ is required to be maximal monotone (potentially multi-valued) instead of pseudo-monotone
%%(and single-valued).}
%
\begin{thm}\label{ru}
Let $V$, $U$, $\V$, $\U$ be as defined in \eqref{space}. Let $A:V\lra V'$
be a maximal monotone operator, let
$\phi:V\lra \R\cup\{+\8\}$ { be} a convex,
lower semi-continuous function such that $D(\phi):=\{v\in V:\phi(v)<+\8\}\neq \emptyset$.
Let $u^0\in U${, and} suppose there exist:
\begin{itemize}
\item $v^0\in D(\phi)$ such that 
\begin{equation}\label{coe}
\lim_{\|v\|_V\ra +\8} \frac{\lg Av,v-v^0\rg_{V',V}+\phi(v)}{\|v\|_V}=+\8,\
\end{equation}

\item $z^0\in U$ such that { for any $v\in V$}
\begin{equation}\label{id2}
 \lg z^0,v\rg+\lg Au^0 ,v\rg_{V',V} +\phi(v)-\phi(u^0) \geq 0.
\end{equation}
\end{itemize}
Then there exists a unique $u\in L^\8(0,T;V)\cap C^0([0,T];U)$ such that $u_t\in L^\8(0,T;U)$,
$u(0)=u^0${,} and
$$\lg u_t(t) ,v(t)-u(t)\rg+ \lg A u(t),v(t)-u(t)\rg_{{ V',V}} +\phi(v(t))-\phi(u(t)) \geq 0$$
for a.e. time $t\in (0,T)$, { and all $v\in V$}.
\end{thm}

%
%{ The next two lemmas ensure that we are under the hypothesis of Theorem \ref{ru}.}
%%We need to check that conditions of Theorem \ref{ru} are satisfied.
%
\begin{lem}
The operator $-\H:V\lra V'$ given by 
\begin{equation}\label{9*}
\lg \H(u),v\rg_{V',V} :=\int_I H(u_{xx})v_x \dd x
\end{equation}
is pseudo-monotone. 
\end{lem}

\begin{proof}
{ To prove that $\H$ is bounded,} given $v\in V$, we observe that
\begin{equation}\label{9**}
\left| \lg \H(u), v\rg_{V',V} \right| = \left| \int_I H(u_{xx})v_x \dd x \right|  \leq
\|H(u_{xx})\|_{L^{3}(I)} \|v_x\|_{L^{3/2}(I)}
\leq c \|u \|_V \|v\|_V,
\end{equation}
where $c$ is a positive constant, thus
$\|\H(u)\|_{V'}\leq c\|u\|_V$.

\bigskip

Consider { $u^n$, $u\in V$ such that} $u^n\rhu u$ { and}
$\limsup_n - \lg \H(u^n), u^n-u\rg_{V',V} \leq 0$.
We need to check { that}
$\lg \H(u), v-u\rg_{V',V}\leq   \liminf_n \lg \H(u^n), v-u^n\rg_{V',V}$ { for all $v\in V$}.
Note that
\begin{align*}
(\fal n) \qquad\lg \H(u), v-u\rg_{V',V} & = \lg \H(u-u^n), v-u\rg_{V',V} + \lg \H(u^n), v-u\rg_{V',V} \\
& =\lg \H(u-u^n), v-u\rg_{V',V} + \lg \H(u^n), v-u^n\rg_{V',V}
+\lg \H(u^n), u^n-u\rg_{V',V},
\end{align*}
where $\lim_n \lg \H(u-u^n), v-u\rg_{V',V} =0 $. 
{ Indeed, since} $u^n\rhu u$ in $V$ (hence in $W^{2,3}(I)$), and the embedding $W^{2,3}(I)\hookrightarrow
W^{1,3}(I)$ is compact, { we have that} $u^n\to u$ in $W^{1,3}(I)$, { and in turn}
\begin{align*}
|\lg \H(u-u^n), v-u\rg_{V',V}| & = \left|\int_I H(u-u^n)_{xx} (v-u)_x \dd x\right| = 
\left|\int_I H(u-u^n)_{x} (v-u)_{xx} \dd x\right|\\ 
& \leq \|H(u-u^n)_x\|_{L^3(I)} \|(v-u)_{xx} \|_{L^{3/2}(I)} \\
&\leq c\|(u-u^n)_x\|_{L^3(I)}
 \|(v-u)_{xx} \|_{L^{3/2}(I)}\to 0,
\end{align*}
for some constant $c>0$.
Moreover{,} $u^n\rhu u$ in $V$ implies { that} $u_x^n\rhu u_x$ in $W^{1,3}(I)$,
and the embedding
$i:W^{1,3}(I) \hookrightarrow C^0([-\pi,\pi];\R)$ (endowed with the $\sup$ norm) is compact.
Hence $u_x^n\rhu u_x$ in $W^{1,3}(I)$ implies 
 $\|u_x^n-u_x\|_{L^\8(I)}\ra 0$, and
$$\lim_n\lg \H(u^n), u^n-u\rg_{V',V} = \lim_n \int_I H(u^n_{xx})(u_x^n-u_x) =0$$
since $\{u^n\}$ is bounded in $V$, and this concludes {\magenta the proof}.
\end{proof}

{{\magenta Note, however, that} the operator $-\H$ is not maximal monotone. To circumvent this {\magenta
difficulty},
let
\begin{align*}
B:V\lra V',&\qquad \lg Bu,v\rg_{V',V}:=\int_I \big[u_{xx}v_{xx}-H(u_{xx})v_x\big]\dd x,\\
\Psi_a:\R\lra (-\8+\8],&\qquad \Psi_a(\xi):=\Phi_a(\xi)-\xi^2/2,\\
\psi:V\to (-\8+\8],&\qquad \psi(u):=\int_I \Psi_a(u_{xx})\dd x.
\end{align*}
Since 
$$\Psi_a''(\xi)=\xi+a+\frac1{\xi+a}-1\geq 1 $$
for any $\xi>-a$, $\Psi_a$ is convex on $(-a,+\8)$. Consequently $\psi$ is convex.

%The following results will be useful.
{\magenta We will use the following properties of the Hilbert transform.}

\begin{enumerate}
\item \cite[Theorem~9.1.3]{BN}
The Hilbert transform $H:L_{\per}^p(I)\lra L_{\per}^p(I)$ is a well-defined,
 linear, bounded operator for any $p\in (1,+\8)$,
where $L_{\per}^p(I):=\{f\in L^p(I):f \text{ is } 2\pi\text{-periodic}\}$.

\item \cite[Theorem~9.1.9]{BN} The Hilbert transform $H:L_{\per}^2(I)\lra L^2_{\per}(I)$ satisfies
$$\|f\|_{L^2(I)}^2=\|H(f)\|_{L^2(I)}^2+\frac1{2\pi}\left(\int_I f\dd x \right)^2$$
for any $f\in L_{\per}^2(I)$.

\item {\magenta Also, we will use the sharp Poincar\'e constant
 for $f\in W^{1,2}_{\per_0}(I)$. To be precise (see \cite[Section~7.7]{HLP}): if} $f\in W^{1,2}_{\per_0}(I)$ then
\begin{equation}\label{PW}
\int_I f^2 \dd x\leq \int_I f_x^2 \dd x,
\end{equation}
where equality holds if and only if $f(\xi)=a\sin \xi+b\cos \xi$ a.e., for some $a,b\in \R$.
\end{enumerate}

\begin{lem}\label{pm}
The operator $B:V\lra V'$ is maximal monotone and coercive.
\end{lem}

\begin{proof}
By construction $B$ is hemi-continuous. To prove monotonicity,
note that
\begin{equation}\label{arma}
|\lg \H u,u\rg_{V',V}|\leq \|H(u_{xx})\|_{L^2(I)}\|u_{x}\|_{L^2(I)}\leq \|u_{xx}\|_{L^2(I)}^2, 
\end{equation}
 since \cite[Proposition~9.1.9]{BN} and 
$\int_I u_{xx}\dd x=0$ give
 $$\|H(u_{xx})\|_{L^2(I)}=\|u_{xx}\|_{L^2(I)}+\frac1{2\pi}\left(\int_I u_{xx}\dd x\right)^2=\|u_{xx}\|_{L^2(I)},$$
while $\|u_{x}\|_{L^2(I)}\leq \|u_{xx}\|_{L^2(I)}$ holds in view of
   \eqref{PW}.
Thus $B$ is monotone and hemi-continuous, hence maximal monotone {\magenta(see Remark \ref{2*})}.
\end{proof}
}

\begin{lem}\label{coercive}
The functionals ${\disp{\F_a(v):=\int_I \Phi_a(v_{xx})\dd x}}$ and ${\psi}$ satisfy
 the coercivity conditions
\begin{equation}\label{coF}
{\lim_{\|v\|_V\ra +\8} \frac{-\lg \H(v),v\rg_{V',V}+\psi(v)}{\|v\|_V}=}
\lim_{\|v\|_V\ra +\8} \frac{-\lg \H(v),v\rg_{V',V}+{\F_a}(v)}{\|v\|_V}=+\8
\end{equation}
\end{lem}

\begin{proof}
Note that
\begin{equation}\label{2}
|\lg \H(v),v\rg_{V',V} |\leq \int_I |H(v_{xx})v_x| \dd x \leq \|v_{xx}\|_{L^{3}(I)} 
\|v_x\|_{L^{3/2}(I)}\leq { c}\|v\|_V
\end{equation}
for some $c>0$.
We consider only functions $v\in V$ such that $v_{xx}+a\geq 0$ a.e. (for the remaining $v$,
it holds $\F_a(v)\equiv +\8$ and the thesis is trivial).
Periodicity{, the }zero-average property of functions of $V${, and Poincar\'e inequality,} imply that
$\|v\|_V\ra+\8$ forces $\|v_{xx}\|_{L^3(I)}\ra +\8$ (and $\|v_{xx}+a\|_{L^3(I)}/\|v_{xx}\|_{L^3(I)}\ra 1$). 
The highest order (and the only relevant) term in $\int_I \Phi_a(v_{xx}) \dd x$ is the cubic term, and
$\int_I (v_{xx}+a)^3 \dd x = \|v_{xx}+a\|_{L^3(I)}^3$.
Poincar\'e inequality gives $\|v\|_V\leq \al \|v_{xx}\|_{L^3(I)}$ for some constant
$\al>0$. Since $\lg \H(v),v-v_0\rg_{V',V} $ is at most quadratic in $\|v_{xx}\|_{L^3(I)}$ (as
$\|v\|_V\ra +\8$), it follows { that}
$$\lim_{\|v\|_V\ra +\8} \frac{-\lg \H(v),v-v_0\rg_{V',V}+\F_a(v)}{\|v\|_V}\geq
\lim_{\|v\|_V\ra +\8} \frac{\|v_{xx}+a\|_{L^3(I)}^3+ \text{lower order terms}}{6\al\|v_{xx}\|_{L^3(I)}}=+\8,$$
{ proving  
$$\lim_{\|v\|_V\ra +\8} \frac{-\lg \H(v),v\rg_{V',V}+{\F_a}(v)}{\|v\|_V}=+\8.$$
The proof for
$$\lim_{\|v\|_V\ra +\8} \frac{-\lg \H(v),v\rg_{V',V}+{\psi}(v)}{\|v\|_V}=+\8$$
is analogous.}
\end{proof}

For future reference, given a mapping $v:[0,T]\lra V$,
 with an abuse of notation we will denote by $v(t,\cdot)$ the function $v(t)$. Hence we will
 often write $v(t,x)$ instead of $v(t)(x)$.

\begin{proof}(of Theorem \ref{main})
Lemma \ref{pm} establishes maximal monotonicity for ${ B}$, while Lemma \ref{coercive} 
{ ensures that }\eqref{coe} {holds}, { and} hypothesis {\eqref{id2} results from \eqref{id}}. 
% {
% Note that inequality 
% $$\lg u_t, v-u\rg -\lg \H(u),v-u\rg +\int_I [\Phi_a(v_{xx})-\Phi_a(u_{xx})]\dd x\geq 0 $$
% can be rewritten as
% $$\lg u_t, v-u\rg -\lg \H(u),v-u\rg + \lg Au,v-u\rg+
%\int_I [\Psi_a(v_{xx})-\Psi_a(u_{xx})]  \dd x\geq 0, $$
% where 
% $$A:V\lra V',\quad \lg Au,v\rg := \int_I u_{xx}v_{xx}\dd x,\qquad \Psi_a(\xi):=\Phi_a(\xi)-\xi^2/2.$$
% Direct computation proves that $\Psi_a$ is convex on $[0,+\8)$, since $\Psi_a''\geq 1$, and
% $A-\H$ is monotone since
% $$|\lg \H(u),u\rg|\leq \|H(u_{xx})\|_{L^2(I)}\|u_{x}\|_{L^2(I)}\leq \|u_{xx}\|_{L^2(I)}^2=\lg Au,u\rg, $$ 
% where the second inequality follows from Wirtinger's inequality and \cite[Proposition~9.1.9]{BN}.
% Direct computation also proves that $A-\H$ is hemicontinuous, that is, for any $u,v\in V$, the mapping
% $t\mapsto \lg(A-\H)u,v \rg$ is continuous in the weak topology of $V'$. 
% Thus $A-\H$ is maximal monotone.
%}
Therefore{,} by Theorem \ref{ru} there exists {a unique} $u:[0,T]\lra V$ such that
$$u\in L^\8(0,T;V)\cap C^0([0,T];U),\qquad u_t\in L^\8(0,T;U),\qquad
u(0)=u^0, $$
and
{\blue
\begin{equation}\label{vi}
\lg u_t,v-u\rg + \lg B(u),v-u\rg_{V',V} + \psi(v)-\psi(u) \geq 0
\end{equation}
for every $v\in V$ and for a.e. $t\in [0,T]$.
Observe that
\begin{align}
\lg &B(u),v-u\rg_{V',V}  + \int_I [\Psi_a(v_{xx})-\Psi_a(u_{xx})]\dd x \notag\\
&= \int_I u_{xx}(v-u)_{xx}\dd x-\lg \H(u),v-u\rg_{V',V}+\int_I\bigg[\Phi_a(v_{xx})-\Phi_a(u_{xx})
-\frac12(v_{xx}^2-u_{xx}^2)\bigg]\dd x,\label{in}
\end{align}
and
$$\frac12\int_I (v_{xx}^2-u_{xx}^2)\dd x = \frac12\int_I (v-u+2u)_{xx}(v-u)_{xx}\dd x
= \int_I u_{xx}(v-u)_{xx}\dd x +\frac12\|(v-u)_{xx}\|_{L^2}^2,$$
hence \eqref{in} becomes
\begin{align*}
\lg B(u),v-u\rg_{V',V} & + \int_I [\Psi_a(v_{xx})-\Psi_a(u_{xx})]\dd x\\
&=
-\lg \H(u),v-u\rg_{V',V}+\int_I[\Phi_a(v_{xx})-\Phi_a(u_{xx})]\dd x-\frac12\|(v-u)_{xx}\|_{L^2}^2.
\end{align*}
Thus the solution $u$ of \eqref{vi1} satisfies also
}
\begin{equation}\label{vi1}
\lg u_t,v-u\rg - \lg \H(u),v-u\rg_{V',V} + \int_I[\Phi_a(v_{xx})-\Phi_a(u_{xx})]\dd x \geq 0 
\end{equation}
{ for every $v\in V$ and for a.e. $t\in [0,T]$.}
We prove that $u$ is also a solution of \eqref{origin} in the weak sense{ of \eqref{3*}},
i.e.{,}
\begin{equation}\label{11*}
\int_0^T \int_I u_t\vph \dd x \dd t = \int_0^T \int_I [H(u_{xx})\vph_{x}{-}\Phi_a'(u_{xx})\vph_{xx}] \dd x \dd t
\end{equation}
{for all $\vph\in C_c^\8((0,T)\times I;\R)$.}
The idea is to test \eqref{vi} (for all $t$ such that it holds) 
with $v=u+\vep(\vph-\bar\vph)$ and $v=u-\vep(\vph-\bar\vph)$, 
where $\bar\vph(t):=\int_I \vph(t,x) \dd x$, take the limit { as}
$\vep\ra0^+$, and integrate in $t$. However 
 it is unclear whether $\Phi_a'(v_{xx})\in L^1(I)$, or
 $v_{xx}+a\geq 0$ a.e. in $x$. 
An {\em ad hoc} construction is required to overcome these difficulties.

\bigskip

{
\noindent{\textbf{Step 1. Integrability of $\log(u_{xx}(t)+a)$.}}
The first step is to prove { that} $\log(u_{xx}(t)+a)\in L^1(I)$ for a.e. $t\in [0,T]$,
and { then show that} $\log(u_{xx}+a)\in L^1(0,T;L^1(I))$.
Fix $\vep\in (0,1)$ and let $v^\vep:=(1-\vep)u(t)$. Using $v^\vep$ in \eqref{vi1}, gives
\begin{align*}
\lg u_t(t),-\vep u(t)\rg - \lg \H(u(t)),-\vep u(t)\rg_{V',V} 
& \geq\int_I \Phi_a(u_{xx}(t))-\Phi_a(v_{xx}^\vep) \dd x\\
&\geq 
 \int_I \vep u_{xx}(t)\Phi_a'((1-\vep)u_{xx}(t)) \dd x , 
\end{align*}
where the last inequality holds since $v_{xx}^\vep=(1-\vep)u_{xx}(t)\geq -(1-\vep)a>-a$,
hence $\Phi_a$ is differentiable in $v_{xx}^\vep(x)$ for a.e. $x\in I${, and also
due to the convexity of $\Phi_a$}.
 By { Lebesgue} monotone convergence { theorem}
\begin{align}
\lg u_t(t),-u(t)\rg - \lg \H(u(t)),- u(t)\rg_{V',V} 
&\geq \lim_{\vep\ra 0^+}\int_I u_{xx}(t)\Phi_a'((1-\vep)u_{xx}(t)) \dd x \notag\\
&=\int_I u_{xx}(t)\Phi_a'(u_{xx}(t)) \dd x.\label{L1ineq}
\end{align}
Note that { for $\xi>-a$,} $\Phi_a'(\xi)=\log(\xi+a)+(\xi+a)^2/2+1$, and { because} $u(t)\in V$,
it follows { that}
\begin{equation}\label{ineq1}
\int_I |u_{xx}(t)(u_{xx}(t)+a)^2|\dd x<+\8,\quad
 \int_{\{u_{xx}(t)+a\geq 1\}} |u_{xx}(t)\log(u_{xx}(t)+a) |\dd x<+\8, 
 \end{equation}
\begin{equation}\label{ineq2}
 \int_{\{u_{xx}(t)\geq -a/2\}} |u_{xx}(t)\log(u_{xx}(t)+a) |\dd x<+\8. 
 \end{equation}
{\magenta Since} $u\in  L^\8(0,T;V)$ and $u_t\in L^\8(0,T;U)$, { we have that}
 $$\lg u_t(t),-u(t)\rg - \lg \H(u(t)),- u(t)\rg_{V',V}<+\8{,} $$ 
which{,} together with} \eqref{L1ineq}{,} \eqref{ineq1}
and \eqref{ineq2}{,} { implies that}
\begin{equation}\label{ineq3}
\int_{J} u_{xx}(t)\log(u_{xx}(t)+a) \dd x<+\8,
\end{equation}
$$J:=\{-a\leq u_{xx}(t)< -a/2\}\cap
\{u_{xx}(t)+a< 1\}.$$
By { definition of $J$, for all $y\in J$} 
$$ u_{xx}(t,y)<0,\quad \log(u_{xx}(t,y)+a)<0,$$
i.e.{,} the integrand
$ u_{xx}(t)\log(u_{xx}(t)+a) $ is nonnegative on $J$.
Since $J\sse\{-a\leq u_{xx}(t)<-a/2\}$, combining with \eqref{ineq3} yields 
 $$\frac{a}2\int_{J} |\log(u_{xx}(t)+a)| \dd x \leq \int_{J} u_{xx}(t)\log(u_{xx}(t)+a) \dd x <+\8, $$
{and so} $\log(u_{xx}(t)+a)\in L^1(I)$.
{ Integrating \eqref{L1ineq}} in time gives
% $$\int_0^T\lg u_t(t),w(t)-u(t)\rg - \lg \H(u(t)),w(t)-u(t)\rg_{V',V} \dd t \geq
%\int_0^T \int_I (u_{xx}(t)-w_{xx}(t))\Phi_a'(w_{xx}(t)) \dd x \dd t. $$
%Testing with $w^\vep(t):=(1-\vep)u(t)$ and letting $\vep\ra 0^+$, the same arguments prove
$\log(u_{xx}+a)\in L^1(0,T;L^1(I))$, with 
$$\int_0^T \int_I |u_{xx}(t)(u_{xx}(t)+a)^2|\dd x \dd t <+\8,\quad
\int_0^T \int_{\{u_{xx}(t)+a\geq 1\}} |u_{xx}(t)\log(u_{xx}(t)+a) |\dd x \dd t<+\8 ,$$
 $$\int_0^T\int_{\{u_{xx}(t)\geq -a/2\}} |u_{xx}(t)\log(u_{xx}(t)+a) |\dd x \dd t<+\8 {,} $$
{ and we conclude that} $u\in L^2(0,T;V)$.

\bigskip

\noindent {\textbf{Step 2. Truncating $u_{xx}(t)$.}} To overcome the issue
that for $\vep {>} 0$, $\vph\in C_c^\8((0,T)\times I;\R)$, the function { $u_{xx}+a+\vep\vph_{xx}$
may} fail to be nonnegative, we construct a sequence $\{u_{xx}^\dt\}$ in the following way:
let
\begin{equation}\label{15**}
u_{xx}^\dt(t,x):=\hat{u}_{xx}^\dt(t,x)-{\red\frac1{2\pi}}\int_I \hat{u}_{xx}^\dt(t,s)\dd s,\qquad
\hat{u}_{xx}^\dt(t,x):=\max \{ u_{xx}(t,x)+a,\dt\}-a.
\end{equation}
{ S}etting
\begin{equation}\label{15*}
{ E^{\dt}(t):=}\{x\in I: \hat{u}_{xx}^\dt(t,x)\neq u_{xx}(t,x)\} \qquad \text{for a.e. } t,
\end{equation}
we have
$\|\hat{u}_{xx}^\dt(t)-u_{xx}(t)\|_{L^p(I)}\leq \dt \L^1(E^\dt(t))^{1/p}$, $p\in [1,3]$.
{ Note that
\begin{equation}\label{15***}
\L^1(E^\dt(t)) \to 0 \quad \text{as } \dt\to 0^+.
\end{equation}}
Since $u(t)\in V\sse W_{\loc}^{2,3}(\R)$, $u_x$ is continuous and $2\pi$-periodic, 
i.e. $\int_I u_{xx}(t,x)\dd x =0$ for a.e. $t$. Thus 
\begin{equation}\label{star}
0 = \int_I u_{xx}(t,x)\dd x \leq \int_I \hat{u}_{xx}^\dt(t,x)\dd x \leq \dt\L^1(E^\dt(t)),
\end{equation}
which gives 
\begin{align*}
\|u_{xx}^\dt(t)-u_{xx}(t)\|_{L^p(I)}&\leq
\|u_{xx}^\dt(t)-\hat{u}_{xx}^{{\dt}}(t)\|_{L^p(I)}+\|\hat{u}_{xx}^\dt(t)-u_{xx}(t)\|_{L^p(I)}\\
&\leq {\red2} \dt \L^1(E^\dt(t))^{1/p},\qquad p\in [1,3].
\end{align*}
{Define} 
$${\disp{u_{x}^\dt(t,x):=\int_{-\pi}^x u_{xx}^\dt(t,y) \dd y + u_x(t,-\pi) - \zeta(t,\dt)}},$$
where
$\zeta(t,\dt) $ is a constant chosen such that ${\disp{\int_I u_x^\dt(t,y) \dd y =0}}$.
Since $\|u_{xx}^\dt(t)-u_{xx}(t)\|_{L^1(I)}\leq 2\dt \L^1(E^\dt(t))$,
it follows { that} $|\zeta(t,\dt)|\leq {\red2\dt \L^1(E^\dt(t))}$.\\
Define also 
$${\disp{u^\dt(t,x):=\int_{-\pi}^x u_{x}^\dt(t,y) \dd y + u(t,-\pi) - \theta(t,\dt),}}$$
where
$\theta(t,\dt) $ is a constant chosen such that ${\disp{\int_I u^\dt(t,y) \dd y =0}}$.
Since $\|u_{x}^\dt(t)-u_{x}(t)\|_{L^1(I)}\leq {\red8}\pi \dt \L^1(E^\dt(t))$,
it follows $|\theta(t,\dt)|\leq 8{\red\pi} \dt \L^1(E^\dt(t))$.
With the above construction{, we now have that}
\begin{enumerate}
\item[{(i)}] $u_{xx}^\dt(t)\geq \dt(1-\L^1(E^\dt(t)){\red/2\pi} )-a$, { where we used \eqref{star};}
\item[{(ii)}] $u_{xx}^\dt(t)\in L^3(I)$ with zero-average on $I$, $u_{x}^\dt(t)\in W_{\per_0}^{1,3}(I)$, 
and $u^\dt(t)\in V$
for a.e. $t${;}
\item[{(iii)}] by Poincar\'e { inequality}, periodicity and the zero-average property of functions in $V${,
we observe that}
$$\|u^\dt(t)-u(t) \|_V \leq \beta\|u_{xx}^\dt(t)-u_{xx}(t)\|_{L^3(I)} \leq 
\beta\dt \L^1(E^\dt(t))^{1/3}(1+ \L^1(E^\dt(t))^{2/3})$$
for some constant $\beta>0$.
  \end{enumerate}

\bigskip

\noindent{\textbf{Step 3. Proof of \eqref{11*}.}}
{ This will be accomplished by testing \eqref{vi} with variations of the form $u^\dt(t)\pm \vep\vph(t)$.}
Fix $\vph\in C_c^\8((0,T)\times I;\R)$, and a time $t$ such that \eqref{vi} holds.
%{ Assume $\sup_x |\vph_{xx}(t,x)|>0$, otherwise the thesis is trivial.
 Two cases apply.

\bigskip

Case 1. { Assume that} there exists $\dt_1>0$ such that $\L^1(E^{\dt_1}(t))=0$. { By \eqref{15**}
and \eqref{15***}, we have that}
$\L^1(E^{\dt}(t))=0$ for any $0<\dt{\leq}\dt_1$ and $u^\dt(t)=u(t)$. { Therefore}
$u_{xx}(t)+a\geq \dt_1$. { Choose $\vep_1>0$} such that
$\vep|\vph_{xx}(t)|<\dt_1/2$ for all $\vep\in(0,\vep_1)$. { W}e consider the variation, for
$\vep\in (0,\vep_1)$,
$$w^\vep(t):=u(t)+\vep(\vph(t)-\bar\vph(t)),\qquad \bar\vph(t){:}={\red\frac1{2\pi}}\int_I \vph(t,x)\dd x.$$
Using $w^\vep(t)$ in \eqref{vi} we get
$$\lg u_t(t),w^\vep(t)-u(t)\rg - \lg \H(u),w^\vep(t)-u(t)\rg_{V',V} + \F_a(w^\vep(t))-\F_a(u(t)) \geq 0,$$
that is,
\begin{equation}\label{case1}
\lg u_t(t),\vep(\vph(t)-\bar\vph(t))\rg 
- \lg \H(u),\vep\vph(t)\rg_{V',V} + \F_a(w^\vep(t))-\F_a(u(t)) \geq 0,
\end{equation}
{ where we used the fact that $\lg\H(u),c\rg_{V',V}=0$ for all constants $c\in \R$ (see \eqref{9*})}.
We need to prove
\begin{equation}\label{c1}
\lim_\vep \frac{\F_a(w^\vep(t))-\F_a(u(t))}{\vep}=\int_I \vph_{xx}(t)\Phi_a'(u_{xx}(t))\dd x. 
\end{equation}
Note that, since $\vep<\vep_1$, both
$u_{xx}(t)+a$ and $u_{xx}(t)+a+\vep\vph_{xx}(t)$ are uniformly bounded away from zero.
%Hence, due to the convexity of $\Phi_a$, it holds
%$$\frac{\F_a(w^\vep(t))-\F_a(u(t))}{\vep} =\frac1\vep
% \int_I \left[\Phi_a(u_{xx}(t)+\vep\vph_{xx}(t)) - \Phi_a(u_{xx}(t))\right] \dd x
%\geq \int_I  \vph_{xx}(t) \Phi_a'(u_{xx}(t))\dd x,$$
%thus
%\begin{equation}\label{c1geq}
%\lim_\vep \frac{\F_a(w^\vep(t))-\F_a(u(t))}{\vep} \geq \int_I  \vph_{xx}(t) \Phi_a'(u_{xx}(t))\dd x.
%\end{equation}
%To prove the opposite inequality, 
We observe that
\begin{align*}
\frac{\F_a(u(t))-\F_a(w^\vep(t))}{\vep} &=\frac1\vep
 \int_I \left[\Phi_a(u_{xx}(t))-\Phi_a(u_{xx}(t)+\vep\vph_{xx}(t)) \right] \dd x\\
&\geq -\int_I  \vph_{xx}(t) \Phi_a'(u_{xx}(t)+\vep\vph_{xx}(t))\dd x.
\end{align*}
Clearly $\vph_{xx}(t) \Phi_a'(u_{xx}(t)+\vep\vph_{xx}(t))$ converges
to $\vph_{xx}(t) \Phi_a'(u_{xx}(t))$ a.e.. { N}ote also that 
$$\Phi_a'(u_{xx}(t)+\vep\vph_{xx}(t)) = \log(u_{xx}(t)+\vep\vph_{xx}(t)+a)+(u_{xx}(t)+\vep\vph_{xx}(t)+a)^2/2, $$
with 
$$u_{xx}(t)+\dt_1/2+a\geq u_{xx}(t)+\vep\vph_{xx}(t)+a\geq \dt_1/2$$ 
due to { the choice of $\dt_1,\vep_1>0$}. Thus
$$\log(u_{xx}(t)+\vep\vph_{xx}(t)+a)+(u_{xx}(t)+\vep\vph_{xx}(t)+a)^2/2 \leq
|\log(\dt_1/2)|+ (u_{xx}(t)+\dt_1/2+a)^2/2\in L^1(I), $$
and, by { Lebesgue} dominated convergence theorem, we have
$$\limsup_\vep \frac{\F_a(u(t))-\F_a(w^\vep(t))}{\vep} 
\geq \lim_\vep -\int_I  \vph_{xx}(t) \Phi_a'(u_{xx}(t)+\vep\vph_{xx}(t))\dd x
=-\int_I  \vph_{xx}(t) \Phi_a'(u_{xx}(t))\dd x,$$
{ or equivalently,
$$\liminf_\vep \frac{\F_a(w^\vep(t))-\F_a(u(t))}{\vep} 
%\geq \lim_\vep -\int_I  \vph_{xx}(t) \Phi_a'(u_{xx}(t)+\vep\vph_{xx}(t))\dd x
\leq\int_I  \vph_{xx}(t) \Phi_a'(u_{xx}(t))\dd x.$$
}
Dividing {\eqref{case1}} by $\vep$
and passing to the limit $\vep\to 0^+$ gives
\begin{align*}
0
&{\red \leq}\lg u_t(t),\vph(t)-\bar\vph(t)\rg 
- \lg \H(u),\vph(t)\rg_{V',V} + \int_I \vph_{xx}(t)\Phi_a'(u_{xx}(t))\dd x.
\end{align*}

\bigskip

{Case 2. Assume {$\L^1(E^{\dt}(t))>0$ for all $\dt>0$}.}
Let $M(\vph):=2\sup_x |\vph_{xx}(t,x)|$,
%$$\vep=\vep(\dt,t):=\dt \L^1(E^\dt(t))^{1/4},\qquad w^\vep(t):=u^\dt(t)+\vep(\vph(t)
%-\bar\vph(t)),\qquad \bar\vph(t):=\int_I \vph(t,x) \dd x.$$
\begin{equation}\label{23*}
\vep=\vep(\vph,\dt,t):={\dt/(1+M(\vph))} ,\ w^\vep(t):=u^\dt(t)+\vep(\vph(t)
-\bar\vph(t)),\ \bar\vph(t):={\red\frac1{2\pi}}\int_I \vph(t,x) \dd x.
\end{equation}
{ Since $\L^1(E^\dt(t))\to 0$ as $\dt\to 0^+$ {(see \eqref{15***})}, and in view of Step 2 (iii), it follows that}
\begin{equation}\label{13*}
\vep=O(\dt),\qquad \|u^\dt(t)-u(t)\|_V=o(\vep) .
\end{equation}
{Taking} $w^\vep(t)$ in \eqref{vi} yields
\begin{align}
\lg u_t(t), u^\dt(t)-u(t) & +\vep(\vph(t)-\bar\vph(t)) \rg -\lg \H(u(t)), 
u^\dt(t)-u(t)+\vep\vph(t)\rg_{V',V} \notag\\
&+\F_a(u^\dt(t)+\vep\vph(t))-\F_a(u(t))\geq 0, \label{inted}
\end{align}
{By the} mean value theorem, we have
\begin{align*}
\F_a(u^\dt(t)+\vep\vph(t))-\F_a(u(t)) &= \int_I \left[\Phi_a(u_{xx}^\dt(t)
+\vep\vph_{xx}(t))-\Phi_a(u_{xx}(t)) \right]\dd x\\
& = \int_{{S(t,\dt)}} ({ u_{xx}^\dt(t)-u_{xx}(t)}+\vep \vph_{xx}(t))\Phi_a'(\vth^\vep(t)) \dd x,
\end{align*}
where 
{ $$S(t,\dt):= \{u_{xx}^\dt(t)+\vep \vph_{xx}(t)\neq u_{xx}(t)\}$$
and}
\begin{equation}\label{mv}
 \min\{u_{xx}^{\dt}(t,x)+\vep\vph_{xx}(t,x), u_{xx}(t,x) \} \leq \vth^\vep(t,x) 
\leq \max\{u_{xx}^{\dt}(t,x)+\vep\vph_{xx}(t,x), u_{xx}(t,x) \}
\end{equation}
 for any $x$. { Next we establish the Lebesgue measurability of $S(t,\dt)\ni x\mapsto \vth(t,x)$.} 
 %
% Note that
% \begin{align*}
% \int_I \left[ \Phi_a(u_{xx}^\dt(t)+\vep\vph_{xx}(t))-\Phi_a(u_{xx}(t)) \right] \dd x &= 
% \int_{S(t,\dt)}
% \left[ \Phi_a(u_{xx}^\dt(t)+\vep\vph_{xx}(t))-\Phi_a(u_{xx}(t)) \right] 
%\\
%&=\int_{S(t,\dt)}
%(u_{xx}^\dt(t)+\vep \vph_{xx}(t)-u_{xx}(t))\Phi_a'(\vth^\vep(t)) \dd x.
% \end{align*}
% Thus it suffices to prove that
%the restriction $\vth^\vep(t,\cdot)_{|S(t,\dt)}$ is Lebesgue measurable.
For $x\in S(t,\dt)$ it holds
 $$\Phi_a(u_{xx}^\dt(t)
+\vep\vph_{xx}(t))-\Phi_a(u_{xx}(t)) =  (u_{xx}^\dt(t)+\vep \vph_{xx}(t)-u_{xx}(t))\Phi_a'(\vth^\vep(t,x)),$$
hence
$$\Phi_a'(\vth^\vep(t,x))=
\frac{\Phi_a(u_{xx}^\dt(t)+\vep\vph_{xx}(t))-\Phi_a(u_{xx}(t)) }{u_{xx}^\dt(t)+\vep \vph_{xx}(t)-u_{xx}(t)}.$$
For $\xi>-a$, $\Phi_a'(\xi)=(\xi+a)^2/2+\log (\xi+a)+1$ is injective, hence we have
$$\vth^\vep(t,x)=(\Phi_a')^{-1}\left(
\frac{\Phi_a(u_{xx}^\dt(t)+\vep\vph_{xx}(t))-\Phi_a(u_{xx}(t)) }{u_{xx}^\dt(t)+\vep \vph_{xx}(t)-u_{xx}(t)}
\right),$$
which proves the Lebesgue measurability of $x\mapsto \vth^\vep(t,x)$ on $S(t,\dt)$.  
\medskip

Dividing by $\vep$ and taking the limit $\dt\ra 0^+$ in \eqref{inted} gives
\begin{align}
  \lg u_t(t), \vph(t)-\bar\vph(t) \rg& -\lg \H(u(t)), 
\vph(t)\rg_{V',V} 
 +  \liminf_\dt\int_{S(t,\dt)}\vph_{xx}(t) \Phi_a'(\vth^\vep(t)) \dd x \notag\\
& +  \frac1\vep \int_{S(t,\dt)}(u_{xx}^\dt(t)-u_{xx}(t)) \Phi_a'(\vth^\vep(t)) \dd x 
\geq 0 ,\label{25}
\end{align}
{ where
 we used the fact that $\|u^\dt(t)-u(t)\|_V=o(\vep)$, and $u$ is Lipschitz in time, and {by \eqref{13*}},
$$\lim_{\dt} \vep^{-1}  \lg u_t(t), u^\dt(t)-u(t) \rg  =0 ,$$
$$\lim_{\dt} \vep^{-1}  |\lg \H(u(t)), u^\dt(t)-u(t)\rg_{V',V} | 
\leq C \lim_\vep\vep^{-1}\|u_{xx}(t)\|_{L^{3}(I)} \|u^\dt(t)-u(t)\|_{V}=0 $$
for some $C>0$.}

\bigskip
{We claim that
\begin{equation}\label{Claim 1}
\lim_\dt \frac1\vep \int_{S(t,\dt)}(u_{xx}^\dt(t)-u_{xx}(t)) \Phi_a'(\vth^\vep(t)) \dd x =0.
\end{equation}
Note that on $I\backslash E^\dt(t)$ it holds $u_{xx}+a\geq \dt$, hence { by \eqref{15**}
and \eqref{star}}, we have
\begin{align*}
\vth^\vep(t)+a &\geq
 \min\{u_{xx}^\dt(t)+\vep\vph_{xx}(t), u_{xx}(t)\}+a\\ 
 &\geq u_{xx}(t)+a-\dt(1/2+ \L^1(E^\dt(t)){\red/2\pi})
 \geq (u_{xx}(t)+a)/3,
 \end{align*}
 {\red for all $\dt$ such that $\L^1(E^\dt(t))\leq \pi/3$, and}
\begin{align*}
\vth^\vep(t)+a &\leq  \max\{u_{xx}^\dt(t)+\vep\vph_{xx}(t), u_{xx}(t)\}+a 
\leq u_{xx}(t)+a+1
 \end{align*}
 {\red for all $\dt\leq 3/2$.}
 Hence 
 \begin{align}
 |\Phi_a'&(\vth^\vep(t))|\notag\\
 &\leq |\log(u_{xx}(t)+a)|+|\log(u_{xx}(t)+a+1)|+\log 3 
 +(u_{xx}(t)+a+1)^2=:g(t)\in L^1(I).
 \label{L1}
 \end{align}
By {\eqref{15**} and \eqref{star}}, 
on $I\backslash E^\dt(t)$ it holds 
\begin{equation}\label{27*}
|u_{xx}^\dt(t)-u_{xx}(t)|\leq \dt \L^1(E^\dt(t)),
\end{equation}
hence
$$\frac1\vep \int_{(I\backslash E^\dt(t))\cap S(t,\dt)}|(u_{xx}^\dt(t)-u_{xx}(t)) \Phi_a'(\vth^\vep(t))| \dd x
\leq 2\L^1(E^\dt(t))\|g(t) \|_{L^1(I)}
\to 0,$$
{ where we have used the definition of $\vep$ as in \eqref{23*}.}
On $E^\dt(t)$ it holds 
$$u_{xx}^\dt(t)+\vep\vph_{xx}(t)+a \geq \dt{\red\bigg(\frac12 - \frac{\L^1(E^\dt(t))}{2\pi}\bigg) }
\geq (u_{xx}+a)/{\red3}, $$
hence
$$\vth^\vep(t)+a \geq
 \min\{u_{xx}^\dt(t)+\vep\vph_{xx}(t), u_{xx}(t)\}+a \geq (u_{xx}+a)/{\red3},$$
thus \eqref{L1} still holds. 
Since $|u_{xx}^\dt(t)-u_{xx}(t)|\leq \dt{\red\big(\frac12 + \frac{\L^1(E^\dt(t))}{2\pi}\big) }$, we have,
{\red under the additional assumption $\dt\leq 3/5$,} 
\begin{equation*}
\frac1\vep \int_{ E^\dt(t)\cap S(t,\dt)}  |(u_{xx}^\dt(t)-u_{xx}(t))\Phi_a'(\vth^\vep(t))| \dd x
 \leq 2\|g(t)\|_{L^1(E^\dt(t))}
\to0,
\end{equation*}
and \eqref{Claim 1} is proven.

%\bigskip

Now we {\magenta show that}
\begin{equation}\label{Claim 2}
\lim_\dt  \int_{S(t,\dt)}\vph_{xx}(t)\Phi_a'(\vth^\vep(t)) \dd x =
\int_I \vph_{xx}(t)\Phi_a'(u_{xx}(t)) \dd x.
\end{equation}
%Note that the moving domain of integration $S(t,\dt)$ prevents direct application of
%dominated convergence theorem. Moreover, both existence and particular expression of the limit $\lim_\dt S(t,\dt)$
%are unclear.
%Thus we need to express the left-hand side term as an integral on $I$. 
By definition of $S(t,\dt)$, { we have}
$$I\backslash S(t,\dt)=\{x:u_{xx}^\dt(t,x)+\vep\vph_{xx}(t,x)-u_{xx}(t,x)=0\},$$ 
thus for any $\dt$ it holds
\begin{align}
0&=\frac1\vep\int_{I\backslash S(t,\dt)} [\Phi_a(u_{xx}^\dt(t)+\vep \vph_{xx}(t))-\Phi_a(u_{xx}(t))] \dd x \notag\\
&= \frac1\vep\int_{I\backslash S(t,\dt)} (u_{xx}^\dt(t)+\vep\vph_{xx}(t)-u_{xx}(t)) \Phi_a'(u_{xx}(t))\dd x.\label{n}
\end{align}
From the construction of $u_{xx}^\dt$, we get
\begin{align*}
\int_{I\backslash S(t,\dt)} &\left|\frac{u_{xx}^\dt(t,x)-u_{xx}(t,x)}{\vep} \Phi_a'(u_{xx}(t))\right| \dd x\\
& = \int_{(I\backslash S(t,\dt))\cap (I\backslash E^\dt(t))} 
\left|\frac{u_{xx}^\dt(t,x)-u_{xx}(t,x)}{\vep} \Phi_a'(u_{xx}(t))\right| \dd x\\
&\quad+\int_{(I\backslash S(t,\dt))\cap E^\dt(t)} 
\left|\frac{u_{xx}^\dt(t,x)-u_{xx}(t,x)}{\vep} \Phi_a'(u_{xx}(t))\right| \dd x\\
& \leq {\red(1+M(\vph))}\Big( \frac{\L^1(E^\dt(t))}{2\pi}\|\Phi_a'(u_{xx}(t))\|_{L^1(I)}+
\Big(1+\frac{\L^1(E^\dt(t))}{2\pi} \Big)
\|\Phi_a'(u_{xx}(t))\|_{L^1(E^\dt(t))}\Big)\to 0,
\end{align*}
{ where we used \eqref{27*} and the fact that on $E^\dt(t)$, $|u_{xx}^\dt-u_{xx}|=O(\dt)$ (see
\eqref{mv}).}

This, together with \eqref{n}, gives
${\disp{\int_{I\backslash S(t,\dt)} \vph_{xx}(t)\Phi_a'(u_{xx}(t)) \dd x \to 0}}$.
Let
$$\tilde\vth^\vep(t,x):=
\left\{\begin{array}{cl}
\vth^\vep(t,x) & \text{if } x\in S(t,\dt),\\
u_{xx}(t,x) & \text{if } x\notin S(t,\dt).
\end{array} \right.$$
Hence 
\begin{align*}
\lim_\dt  \int_{S(t,\dt)}\vph_{xx}(t)\Phi_a'(\vth^\vep(t)) \dd x  &= 
\lim_\dt  \int_{S(t,\dt)}\vph_{xx}(t)\Phi_a'(\vth^\vep(t)) \dd x +
\lim_\dt   \int_{I\backslash S(t,\dt)}\vph_{xx}(t)\Phi_a'(u_{xx}(t)) \dd x\\
  &=\lim_\dt \int_{I}\vph_{xx}(t)\Phi_a'(\~\vth^\vep(t)) \dd x,
  \end{align*}
thus \eqref{Claim 2} is equivalent to { proving that}
\begin{equation}\label{*}
\lim_\dt  \int_{I}\vph_{xx}(t)\Phi_a'(\~\vth^\vep(t)) \dd x =
\int_I \vph_{xx}(t)\Phi_a'(u_{xx}(t)) \dd x.
\end{equation}
By construction, $u_{xx}^\dt(t)+\vep\vph_{xx}(t)\to u_{xx}(t)$ a.e., hence
$\~\vth^\vep(t)\to u_{xx}(t)$ a.e. { Therefore, \eqref{*} follows from \eqref{L1}
and Lebesgue dominated convergence theorem.}
%$$|\Phi_a'(\~\vth^\vep(t))-\Phi_a'(u_{xx}(t))|\leq |1_{S(t,\dt)}\Phi_a'(\vth^\vep(t))|
%\overset{\eqref{L1}}{\leq}|\log(u_{xx}(t)+a)|+\log 3+(u_{xx}(t)+a+1)^2\in L^1(I), $$
%by dominated convergence theorem we have
%$$ \int_{I}|\vph_{xx}(t) [ \Phi_a'(\~\vth^\vep(t)) 
%-\Phi_a'(u_{xx}(t))]| \dd x \leq \sup_x |\vph_{xx}(t,x)|\int_I|\Phi_a'(\~\vth^\vep(t))-
%\Phi_a'(u_{xx}(t))|\dd x
%\to 0,$$
%which proves Claim 2.
} 

%\bigskip

In view of \eqref{Claim 1} and \eqref{Claim 2}, passing to the limit $\dt\to 0^+$ in \eqref{25} we get
%\begin{align}
%  \lg u_t(t), \vph(t)-\bar\vph(t) \rg& -\lg \H(u(t)), 
%\vph(t)\rg_{V',V} 
% +  \int_I \vph_{xx}(t)\Phi_a'(u_{xx}(t)) \dd x
%\geq 0 .\label{26}
%\end{align}
%
$$ \int_I  u_t(t)(\vph(t) -\bar\vph(t))\dd x \geq  
\int_I \big[H(u_{xx}(t))\vph_x(t)-\Phi_a'(u_{xx}(t))\vph_{xx}(t)\big] \dd x. $$
 The above argument can be repeated for any $t$ in 
$$\{t\in (0,T): \eqref{vi}\text{ holds, }  \log (u_{xx}(t)+a)\in L^1(I),\ u(t)\in V\}, $$
which has full measure, yielding
$$\int_I  u_t(t)(\vph(t)-\bar\vph(t)) \dd x  \geq
 \int_I \big[H(u_{xx}(t))\vph_x(t)-\Phi_a'(u_{xx}(t))\vph_{xx}(t)\big] \dd x  
 \qquad \text{for a.e. } t.$$
Integrating in time gives
\begin{equation}\label{ge pre}
\int_0^T \int_I  u_t(t)(\vph(t)-\bar\vph(t)) \dd x  \dd t \geq \int_0^T
 \int_I \big[H(u_{xx}(t))\vph_x(t)-\Phi_a'(u_{xx}(t))\vph_{xx}(t)\big] \dd x \dd t .
 \end{equation}
Since $u$ is Lipschitz in time and $\vph$ { is} smooth, we have sufficient regularity
to integrate by parts, hence
$$\int_0^T \int_I u_t(t) \bar\vph(t) \dd x \dd t 
=-\int_0^T \bar\vph_t(t)\Big( \int_I u(t) \dd x\Big) \dd t =0,$$
and \eqref{ge pre} { becomes}
\begin{equation}\label{ge}
\int_0^T \int_I  u_t(t)\vph(t) \dd x  \dd t \geq \int_0^T
 \int_I \big[H(u_{xx}(t))\vph_x(t)-\Phi_a'(u_{xx}(t))\vph_{xx}(t)\big] \dd x \dd t .
 \end{equation}
{ Replacing $\vph$ with $-\vph$ in \eqref{ge}, we conclude \eqref{11*}.}
% The above construction can be also repeated with $w^\vep(t):=u(t)-\vep(\vph(t)-\bar\vph(t))$,
% which gives
% \begin{equation}\label{le}
% -\int_0^T \int_I  u_t(t)\vph(t) \dd x  \dd t \geq -\int_0^T
% \int_I H(u_{xx}(t))\vph_x(t)-\Phi_a'(u_{xx}(t))\vph_{xx}(t) \dd x \dd t .
% \end{equation}
%Combining \eqref{ge} and \eqref{le} concludes the proof.
\end{proof}

\bigskip
%
%\subsection*{Acknowledgments.} 
%XYL warmly thanks the Center of Nonlinear Analysis (NSF grant DMS-0635983),
% and acknowledges the support by ICTI and FCT (grant UTA\_CMU/MAT/0007/2009).
% Irene Fonseca and Giovanni Leoni's acknowledgments

\section*{Acknowledgements}
The authors warmly thank the Center for Nonlinear Analysis (NSF PIRE Grant No.
OISE-0967140), where part of this research was carried 
out. The research of I. Fonseca was partially funded by the National Science
Foundation under Grant No.
DMS-0905778, DMS-1411646, and that of G. Leoni under Grant No. DMS-1007989,
DMS-1412095. The research of X.Y. Lu was supported 
by the Funda\c{c}$\tilde{\text{a}}$o para a Ci\^{e}ncia e a Tecnologia (Portuguese Foundation for
Science and Technology) through the
Carnegie Mellon--Portugal Program under Grant SFRH/BD/35695/ 2007.

\end{document}